\documentclass[11pt]{amsart}
\usepackage{xcolor}

\usepackage[margin=1in]{geometry}
\usepackage{color}
\usepackage{amssymb}
\usepackage{amsmath}
\usepackage{amsthm}
\usepackage{url}
\usepackage{framed}
\usepackage{enumitem}

\newtheorem{theorem}{Theorem}
\newtheorem{proposition}[theorem]{Proposition}
\newtheorem{lemma}[theorem]{Lemma}
\newtheorem{corollary}[theorem]{Corollary}

\theoremstyle{definition}

\newtheorem{example}[theorem]{Example}
\newtheorem{remark}[theorem]{Remark}
\newcommand{\tr}{\mathop\mathrm{tr}}

\newcommand{\C}{{\mathbb{C}}}

\newcommand{\Hc}{\mathcal{H}}
\newcommand{\CHI}{\hbox{\raise .4ex \hbox{$\chi$}}}

\newcommand{\Fc}{{\mathcal{F}}}

\newcommand{\R}{{\mathbb{R}}}

\begin{document}

\title{Frame Scalings: A Condition Number Approach}


\author{Peter Casazza}
\address{Department of Mathematics, University of Missouri, Columbia}
\email{casazzap@missouri.edu}
\author{Xuemei Chen}
\address{Department of Mathematics and Statistics, University of San Francisco} \email{xchen@math.usfca.edu}

\thanks{The authors were supported by
 NSF DMS 1307685; NSF ATD 1321779}
\maketitle

\begin{abstract}
Scaling frame vectors is a simple and noninvasive way to construct tight frames. However, not all frames can be modifed to tight frames in this fashion, so in this case we explore the problem of finding the best conditioned frame by scaling, which is crucial for applications like signal processing. We conclude that this problem is equivalent to solving a convex optimization problem involving the operator norm, which is unconventional since this problem was only studied in the perspective of Frobenius norm before. We also further study the Frobenius norm case in relation to the condition number of the frame operator, and the convexity of optimal scalings.
\end{abstract}

\vspace{0.1in}


\vspace{0.1in}


\section{Introduction}

A family of vectors $\Phi=\{\varphi_i\}_{i=1}^M$ is a \emph{frame} in an $N$-dimensional Hilbert space $\mathcal{H}_N$   if there are constants $0 < A \leq B < \infty$ so that for all $x \in \mathcal{H}_N$,
\begin{equation*} 
A \| x \|^2 \leq \sum_{i=1}^N |\langle x,\varphi_i\rangle|^2 \leq B \|x\|^2.
\end{equation*}
 The largest $A$ and smallest $B$ satisfying these inequalities are called  \emph{lower and upper frame bounds}, respectively. One often also writes $\Phi$ for the $N\times M$ matrix whose $i$th column is the vector $\varphi_i$. When $A=B$, the frame is called an \emph{$A$-tight frame}. Furthermore, $A=B=1$ produces a \emph{Parseval frame}. In the sequel, the set of frames  with $M$ vectors in $\Hc_N$ will be denoted by $\Fc(M,N)$. It is well known that $\Phi$ is $A$-tight if and only if
 \begin{equation}\label{equ:tight}
S:= \Phi\Phi^*=\sum_{i=1}^M\varphi_i\varphi_i^*=AI_N,
 \end{equation}
 where $I_N$ is the identity matrix in $\Hc_N$, and
 $S$ is the {\it frame operator} of the frame.
  We refer to \cite{co03} for an introduction to frame theory and to \cite{ck12}  for an overview of the current research in the field.

Frames have traditionally played a significant role in the theory of signal processing,  but today they have found application
to packet based network communication
 \cite{CCHKP11,FM2012}, wireless sensor networks 
\cite{CFHWZ12,CFMWZ11,CHKWZ12,CP13},
distributed processing
\cite{CCHKP11}, quantum information theory, bio-medical
engineering \cite{BMFCOK2012,MBFOK2012}, compressed sensing \cite{BCC12, CWW14}, fingerprinting \cite{MQ},
spectral theory \cite{CFMPS2011,FMPS2012,FMP2012},
and much more.

Some of the applications of frames result from their ability to deliver redundant, yet stable expansions. The redundancy of a frame is typically utilized by applications which may require robustness of the frame coefficients to noise, erasures, quantization, etc. In this setting tight frames can give fast convergence and recovery. It is known that unit norm tight frames are characterized in terms of 
the frame potential \cite{BF}. There have been various works on constructions of tight frames \cite{BW13, CFMPS2011, CFMWZ11, CHKWZ12, LHH13, VW08}. However, it is desirable to construct tight frames by just scaling each frame vector as it is noninvasive, and frame properties such as erasure resilience or sparse expansions are left untouched by this modification.

This procedure is called {\it frame scaling}. To be specific, a frame $\Phi = \{\varphi_i\}_{i=1}^M$ for $\Hc_N$ is called \emph{scalable} if there exist scalars $\{s_i\}_{i=1}^M$ such that $\{s_i\varphi_i\}_{i=1}^M$ is a tight frame for $\Hc_N$. By the nature of scaling, if a frame is scalable, there exist scalars such that the scaled frame is a Parseval frame. So by \eqref{equ:tight}, a frame is scalable if and only if there exists $c_i\geq0$ such that
$$I_N=\sum_{i=1}^Mc_i\varphi_i\varphi_i^*.$$
The $c_i$ here corresponds to each scalar as $c_i=|s_i|^2\geq0$. The notion of a {\it scalable frame} was first introduced in \cite{kopt13}. 
In \cite{kopt13},
characterizations of scalable frames,
both of functional analytic and geometric type were derived in the infinite as well as finite dimensional
settings. 
The work \cite{cc12} considers the complex case as well, and it was shown that the set of all possible sequences of scalars is the convex hull of minimal scalars. The paper \cite{CO15} focuses on the numerical algorithms to find different scalings with different purposes. Recently, the work \cite{CGOPW15} studies the case when a frame is not scalable by measuring 
\begin{equation}\label{equ:fro}\min_{c_i\geq0}\|I_N-\sum_{i=1}^Mc_i\varphi_i\varphi_i^*\|_F
\end{equation}
using the minimal ellipsoid of the convex hull of the frame vectors, where $\|\cdot\|_F$ is the Frobenius norm. The theory of frame scalings has also been extended to matrices with Laurent polynomials (with applications to construct tight wavelet filter banks) \cite{HO15}, and probabilistic frames \cite{LO15}.

When a frame is not scalable, we wish to find scalars such that $\{s_i\varphi_i\}_{i=1}^M$ is as tight as possible. This should naturally mean that $\{s_i\varphi_i\}_{i=1}^M$ is the best conditioned in the sense that the ratio of upper and lower frame bounds is the closest to 1. However, it is not clear whether solving \eqref{equ:fro} gives the best conditioned frame. 

In this note, we study the scalability of frames by  $s_i$ so that $\{s_i\varphi_i\}_{i=1}^M$ is best conditioned. For a given square matrix $T$, we define $$\displaystyle\text{cond}(T):=\frac{\text{biggest singular value of $T$}  }{\text{smallest singular value of $T$} }.$$ We include the singular case where condition number is $\infty$. We wish to solve 
\begin{equation}\label{equ:cond}
\min_{c_i\geq0}\text{cond}\left(\sum_{i=1}^Mc_i\varphi_i\varphi_i^*\right),\end{equation}
where $\text{cond}\left(\sum_{i=1}^Mc_i\varphi_i\varphi_i^*\right)$ is exactly the ratio of the upper and lower frame bounds of the scaled frame $\{s_i\varphi_i\}_{i=1}^M$. For convenience of discussion, we include ``frames'' that do not span, whose condition number is $\infty$.

The first contribution of this paper is to establish an equivalence between \eqref{equ:cond} and the following problem:
\begin{equation}\label{equ:operator}\min_{c_i\geq0}\|I_N-\sum_{i=1}^Mc_i\varphi_i\varphi_i^*\|_2,
\end{equation}
where $\|\cdot\|_2$ is the operator norm of a matrix. We prove this equivalence in Theorem \ref{thm:equiv} of Section \ref{sec:operator}. This means that rather than solving the scalability problem with the Frobenius norm, using the operator norm is more efficient if one wants to find the best conditioned frame by scaling, which is the original goal of scaling. Moreover, it is not very clear how one can solve problem \eqref{equ:cond} at first glance. With Theorem \ref{thm:equiv}, we have converted it to a convex programming Problem. Some properties of the minimizer of \eqref{equ:operator} are also established in Section \ref{sec:operator}.

The second contribution of the paper is to study how solving \eqref{equ:fro} is related to solving \eqref{equ:cond}. Unfortunately, we show in Section \ref{sec:invert} that the best scaled ``frame" from \eqref{equ:fro} may not even span
the space (hence is not a frame). Theorem \ref{thm17} lists a sufficient condition for when the best scaling corresponds to a frame in $\Hc_N$. In the end, we further study the convexity of all the optimal scalings (a polytope), and show that the vertices of this polytope are the so called minimal optimal scalings. This is interesting in its own right.

\section{Notations and convexity of optimal scalings}\label{sec:note}
 Given $\Phi=\{\varphi_i\}_{i=1}^M\in\Fc(M,N)$, we define $F=\left(|\langle\varphi_i,\varphi_j\rangle|^2\right)$, which is the Gram matrix of the frame of the outer products $\{\varphi_i\varphi_i^*\}_{i=1}^M$. Let $g=(\|\varphi_1\|^2, \|\varphi_2\|^2, \cdots, \|\varphi_M\|^2)$ be the vector of norms squared.


Some of the notations are common for both \eqref{equ:fro} and \eqref{equ:operator}, so $\|\cdot\|$ could be either the operator norm or the Frobenius norm in the following.  Let $$\min_{c_i\geq0}\|I_N-\sum_{i=1}^Mc_i\varphi_i\varphi_i^*\|=\|I_N-\sum_{i=1}^Mc_i'\varphi_i\varphi_i^*\|=\|I_N-T_{\Phi}\|.$$
We call $(c_i')_{i=1}^M$  an optimal scaling and $T_{\Phi}=\sum_{i=1}^Mc_i'\varphi_i\varphi_i^*$ an optimal operator.
We call $\mathcal{O}_{\Phi}=\{(c_i')_{i=1}^M: \min_{c_i\geq0}\|I_N-\sum_{i=1}^Mc_i\varphi_i\varphi_i^*\|=\|I_N-\sum_{i=1}^Mc_i'\varphi_i\varphi_i^*\|\}$ the set of optimal scalings. Moreover, we define  $(c_i')_{i=1}^M\in\mathcal{O}_{\Phi}$ to be a minimal optimal scaling if no proper subset of $\{\varphi_i\varphi_i^*:c_i'>0\}$ can span $T_{\Phi}$ with nonnegative coefficients. 


$T_{\Phi}$ is not necessarily unique for the operator norm (See Example \ref{exa:not}), in which case $T_{\Phi}$ will mean the set of optimal operators. So in the definition of minimal optimal scaling above, ``span $T_{\Phi}$'' means span any operator in $T_{\Phi}$. In fact, the concept of minimal optimal scaling is much more meaningful in the Frobenius norm case as we will see in Section \ref{sec:conv}.

The following theorem states that $\mathcal{O}_{\Phi}$ is a convex set. But for the Frobenius case, we can actually say much more, see Section \ref{sec:conv}.

\begin{theorem}
With either operator norm or Frobenius norm, the set of optimal scalings $\mathcal{O}_{\Phi}$ is a convex set. 
\end{theorem}
\begin{proof}
Let $(c_i')_{i=1}^M, (d_i')_{i=1}^M \in\mathcal{O}_{\Phi}$ be two optimal scalings, and $T_1=\sum_{i=1}^M c_i'\varphi_i\varphi_i^*, T_2=\sum_{i=1}^Md_i'\varphi_i\varphi_i^*$ be the corresponding optimal operators. For any $a,b\geq0, a+b=1$, we can prove $(ac_i'+bd_i')_{i=1}^M$ is an optimal scaling too since
\begin{align*}
&\|I_N-\sum_{i=1}^M(ac_i'+bd_i')\varphi_i\varphi_i^*\|=\|I_N-(aT_1+bT_2)\|\\
\leq&a\|I_N-T_1\|+b\|I_N-T_2\|=\min_{c_i\geq0}\|I_N-\sum_{i=1}^Mc_i\varphi_i\varphi_i^*\|.
\end{align*}
\end{proof}

\section{Minimizing with the operator norm}\label{sec:operator}
This section focuses on problem \eqref{equ:operator}, which uses the operator norm for the scaling problem. Once again, let 
\begin{equation}\label{equ:op2}
\min_{c_i\geq0}\|I_N-\sum_{i=1}^Mc_i\varphi_i\varphi_i^*\|_2=\|I_N-T_{\Phi}\|_2.
\end{equation}
The optimal operator $T_{\Phi}$ does not need to be unique here. 
\begin{example}\label{exa:not}
Let $\varphi_1=(1,\frac{1}{\sqrt{2}},0), \varphi_2=(\frac{1}{\sqrt{2}},1,0), \varphi_3=(0,0,1)$, so $\{\varphi_i\}_{i=1}^3$ is a frame of $\R^3$. For any $c_i\geq0$, it is straightforward to calculate the eigenvalues of the operator $\sum_{i=1}^3c_i\varphi_i\varphi_i^*$. One can obtain, without much effort, that \eqref{equ:op2} is minimized when $c_1=c_2=2/3
$, and $c_3$ takes on any value in $[1-2\sqrt{2}/3,1+2\sqrt{2}/3]$. With such selection of $c_i's$, $\sum_{i=1}^3c_i\varphi_i\varphi_i^*$ have
 eigenvalues $1-2\sqrt{2}/3, c_3, 1+2\sqrt{2}/3$, which result in different operators with different $c_3$.\end{example}

The following proposition is obvious.
\begin{proposition}\label{pclem1}
If a positive semi-definite matrix $T$ has eigenvalues
$\lambda_1 \ge \lambda_2 \ge \cdots \ge \lambda_N \geq0$, then
\[ \|I_N - T\|_2= \max\{|1-\lambda_1|,|1-\lambda_N|\}.\]
\end{proposition}

The next proposition is essential in proving that solving \eqref{equ:operator} optimizes the condition number. It essentially says that the optimal operator comes to a balance when the largest and the smallest eigenvalue have equal distance to 1.
\begin{proposition}\label{pro:1N}
Let $\Phi = \{\varphi_i\}_{i=1}^M\in\Fc(M,N)$ 
be a frame for $\Hc_N$.
If $T=\sum_{i=1}^Mc_i\varphi_i\varphi_i^*$ has eigenvalues $\lambda_1 \ge \lambda_2 \ge 
\cdots \ge \lambda_N \geq0$ and $c=\frac{2}{\lambda_1+\lambda_N}$, then
\[ \|I-cT\|_2 = c\lambda_1-1=1-c\lambda_N.\]
Moreover, if $T=T_{\Phi}$ then
\[ \|I_N-T\|_2=\lambda_1-1=1-\lambda_N,\]
and consequently $\lambda_1+\lambda_N=2$.
\end{proposition}
\begin{proof}
We check
\[ c\lambda_N = \frac{2}{\lambda_1+\lambda_N}\lambda_N
= \frac{2}{\frac{\lambda_1}{\lambda_N}+1}\le 1,
\]
and
\[ c\lambda_1 = \frac{2}{\lambda_1+\lambda_N}\lambda_1
= \frac{2}{1+\frac{\lambda_N}{\lambda_1}}\ge 1,
\]
Also,
\[ c\lambda_1+c\lambda_N = c(\lambda_1+\lambda_N) =2,\]
so that $ c\lambda_1-1=1-c\lambda_N.$
Now, 
\[ \|I_N-cT\|_2=\max\{|1-c\lambda_1|,|1-c\lambda_N|\}=c\lambda_1-1=1-c\lambda_N.\]

Now assume $T=T_{\Phi}$ and we check two cases.
\vskip10pt
\noindent {\bf Case 1:}  $\lambda_1+\lambda_N > 2$.

In this case, $c < 1$ and $\lambda_1>1$. So
\[ \|I_N-cT\|_2=c\lambda_1-1 < |\lambda_1-1|
\leq \|I_N-T\|_2,\] which is a contradiction to \eqref{equ:op2}.
\vskip10pt
\noindent {\bf Case 2:}  $\lambda_1+\lambda_N < 2$.

In this case, $c>1$ and $\lambda_N<1$. So
\[ \|I_N-cT\|_2= 1-c\lambda_N < |1-\lambda_N|
\le \|I_N-T\|_2,\]  which is a contradiction to \eqref{equ:op2}.

 So we must have $c=1$ and therefore $\lambda_1\geq1\geq\lambda_N$.
\end{proof} 

The proof of Proposition \ref{pro:1N} immediately implies
\begin{corollary}\label{cor:balance}
If $T$ has eigenvalues $\lambda_1\geq\cdots\geq\lambda_N\geq0$, let $c=\frac{2}{\lambda_1+\lambda_N}$, then
$$\|I_N-cT\|_2\leq\|I_N-T\|_2,$$ and equality holds if and only if $c=1$.
\end{corollary}

Now we are ready to prove the main theorem.
\begin{theorem}\label{thm:equiv}
Problem \eqref{equ:cond} and Problem \eqref{equ:operator} are equivalent in the sense that
\begin{itemize}
\item[(a)] Solving \eqref{equ:operator} gives the minimal condition number among $\{\sum_{i=1}^M c_i\varphi_i\varphi_i^*: c_i\geq0\}$.
\item[(b)] If $\{d_i\}_{i=1}^M=\arg\min_{c_i\geq0}\text{cond}\left(\sum_{i=1}^Mc_i\varphi_i\varphi_i^*\right)$, then the scalars $\{cd_i\}_{i=1}^M$ achieves the minimum in \eqref{equ:operator} for some $c>0$.
\end{itemize} As a consequence of (a), if $$\min_{c_i\geq0}\|I_N-\sum_{i=1}^Mc_i\varphi_i\varphi_i^*\|_2=\|I_N-T_{\Phi}\|_2,$$ then $T_{\Phi}$ has the smallest condition number among $\{\sum_{i=1}^Mc_i\varphi_i\varphi_i^*: c_i\geq0\}$, and the condition number of the optimal operator has an upper bound as $$\text{cond}(T_{\Phi})\leq \text{cond}(\Phi\Phi^*).$$
\end{theorem}
\begin{proof}
(a) Assume $T_{\Phi}$ has eigenvalues $\lambda_1 \ge \lambda_2 \ge 
\cdots \ge \lambda_N\geq0$.
By Proposition \ref{pro:1N}, $\lambda_1 \ge 1 \ge \lambda_N$ and $\lambda_1 -1 = 1-\lambda_N$.  For arbitrary $c_i\geq0$, let  $R=\sum_{i=1}^Mc_i\varphi_i\varphi_i^*$ have eigenvalues
$\mu_1 \ge \mu_2 \ge \cdots \ge \mu_N\geq0$.  Letting 
$c=\frac{2}{\mu_1+\mu_N}$, we have by Proposition \ref{pro:1N},
\[ \|I_N-cR\|_2= c\mu_1-1 = 1-c\mu_N.\]
Now, 
\[ \lambda_1-1 = \|I_N-T_{\Phi}\|_2 \le \|I_N-cR\|_2 = c\mu_1-1.\]
Hence, $1\le \lambda_1 \le c\mu_1$. Similarly,
\[ 1-\lambda_N = \|I_N-T_{\Phi}\|_2\le \|I_N-cR\|_2 = 1-c\mu_N,\]
and hence $1\ge \lambda_N \ge c\mu_N$.  It follows immediately that
\[ \text{cond}(T_{\Phi})= \frac{\lambda_1}{\lambda_N}\leq \frac{c\mu_1}{c\mu_N}=\text{cond}(R).\] 

This shows that solving \eqref{equ:operator} gives the minimal condition number among all scalings $c_i$. 
\vskip10pt
(b) Suppose $T=\sum_{i=1}^Md_i\varphi_i\varphi_i^*$ has the smallest condition number after solving \eqref{equ:cond}, and $T$ has eigenvalues $\lambda_1\geq\cdots\geq\lambda_N>0$. For arbitrary $c_i\geq0$, let $R=\sum_{i=1}^Mc_i\varphi_i\varphi_i^*$ have eigenvalues
$\mu_1 \ge \mu_2 \ge \cdots \ge \mu_N\geq0$, and let $c=\frac{2}{\lambda_1+\lambda_N}, d=\frac{2}{\mu_1+\mu_N}$.  So
\begin{align*}
&\text{cond}(cT)\leq\text{cond}(dR)\\
\Rightarrow&\frac{c\lambda_1}{c\lambda_N}\leq \frac{d\mu_1}{d\mu_N}\quad\Rightarrow \frac{2-c\lambda_N}{c\lambda_N}\leq \frac{2-d\mu_N}{d\mu_N}\\
\Rightarrow& 1-c\lambda_N\leq1-d\mu_N\quad\Rightarrow\|I_N-cT\|_2\leq \|I_N-dR\|_2.
\end{align*} By Corollary \ref{cor:balance},$$\|I_N-cT\|_2\leq \|I_N-dR\|_2\leq\|I_N-R\|_2.$$
\end{proof}

\begin{remark}From the proof of Proposition \ref{pro:1N}, Corollary \ref{cor:balance} and Theorem \ref{thm:equiv}, Theorem \ref{thm:equiv} also holds for the case when $\Phi$ is not a frame (do not span $\Hc_N$).\end{remark}

If we do start with an actual frame $\Phi$, since the condition number of $T_{\Phi}$ has a finite upper bound, we immediately have
\begin{corollary}
Given $\Phi\in\Fc(M,N)$, using the operator norm, any optimal operator $T_{\Phi}$, as in $\min_{c_i\geq0}\|I_N-\sum_{i=1}^Mc_i\varphi_i\varphi_i^*\|_2=\|I_N-T_{\Phi}\|_2$,  is invertible.
\end{corollary}

Being invertible means that the smallest eigenvalue $\lambda_N>0$, so we can easily get an upper bound of $\|I_N-T_{\Phi}\|_2$ as 
$$\|I_N-T_{\Phi}\|_2=1-\lambda_N<1.$$
The following theorem shows that this upper bound is tight.

\begin{theorem}\label{thm:1}
For any fixed $M,N$, and any $\Phi\in\Fc(M,N)$,
$$\min_{c_i\geq0}\|I_N-\sum_{i=1}^Mc_i\varphi_i\varphi_i^*\|_2=\|I_N-T_{\Phi}\|_2<1.$$ Moreover, given an arbitrary $\epsilon>0$, there exists a frame $\Psi=\{\psi_i\}_{i=1}^M$ such that 
$$\min_{c_i\geq0}\|I_N-\sum c_i\psi_i\psi_i^*\|_2=1-\epsilon.$$
\end{theorem}

To prove the tightness part of this theorem, we need a few lemmas first.

\begin{lemma}
Given $M\ge N$, there is a universal constant $K>0$ so
that whenever $\{\varphi_i\}_{i=1}^M$ is a unit norm
frame in $\Hc_N$ and
\[\|I_N-\sum_{i=1}^M c_i'\varphi_i\varphi_i^*\|_2=
\min_{c_i\geq0}\|I_N-\sum c_i\varphi_i\varphi_i^*\|_2,\]
then $c_i' \le K$ for all $i=1,2,\ldots,M.$
\end{lemma}

\begin{proof}
Assume the optimal operator $T_{\Phi}$ has eigenvalues $\lambda_1 \ge \cdots\geq\lambda_N$.
By Proposition \ref{pro:1N}, 
\[ \lambda_1 \le \lambda_1+\lambda_N =2.\]
Therefore
\[ c_i'\le \sum_{i=1}^M c_i' =\sum_{i=1}^Mc_i'\|\varphi_i\|_2^2 =\text{Tr}\left(\sum_{i=1}^M c_i'\varphi_i\varphi_i^*\right)= \sum_{i=1}^N\lambda_i
\le 2N,\] 
\end{proof}

\begin{lemma}\label{lem:continuous}
The function
\[ f(\Phi)=f(\varphi_1,\varphi_2,\ldots,\varphi_M)=\min_{c_i\geq0}\|I_N-\sum c_i\varphi_i\varphi_i^*\|_2,\]
defined on $N\times M$ matrices with unit norm columns,
is  continuous with respect to $\Phi$.
\end{lemma}

\begin{proof}
Given a sequence of $N\times M$ matrices with unit norm columns  $\Phi(n)=\{\varphi_i(n)\}_{i=1}^M$ such that
\[ \lim_{n\rightarrow \infty}\Phi(n) = \Phi,\]
we need to prove $\lim_{n\rightarrow\infty} f(\Phi(n))=f(\Phi)$.

Choose $\{d_i(n)\}_{i=1}^M$ so that
\[ \|I_N-\sum_{i=1}^Md_i(n)\varphi_i(n)\varphi_i(n)^*\|_2=
\min_{c_i\geq0}\|I_N-\sum_{i=1}^M c_i\varphi_i(n)\varphi_i(n)^*\|_2=f(\Phi(n)).\]
Since the $\{d_i(n)\}$ are uniformly bounded by the
previous Lemma, by switching to a subsequence we may
assume
\[ \lim_{n\rightarrow \infty}d_i(n)=d_i\mbox{ for all }
i=1,2,\ldots,M.\]

Now, if 
\begin{equation}\label{equ:goal}
\|I_N-\sum_{i=1}^M d_i\varphi_i\varphi_i^*\|_2=f(\Phi), 
\end{equation}
we are done. For any scalar $c_i\geq0$,

\begin{align*}
\|I_N-\sum_{i=1}^M c_i\varphi_i\varphi_i^*\|_2&=\lim_{n\rightarrow\infty}\|I_N-\sum_{i=1}^M c_i\varphi_i(n)\varphi_i(n)^*\|_2\\
&\geq\lim_{n\rightarrow\infty}\|I_N-\sum_{i=1}^M d_i(n)\varphi_i(n)\varphi_i(n)^*\|_2\\&=\|I_N-\sum_{i=1}^M d_i\varphi_i\varphi_i^*\|_2,
\end{align*}
which means that the scalars $(d_i)_{i=1}^M$ achieves the minimum, hence \eqref{equ:goal}.
\end{proof}

\begin{proof}[Proof of Theorem \ref{thm:1}]
Let $f(\Phi)$ be as defined in Lemma \ref{lem:continuous}.


Let $e_1=[1,0,\cdots,0]\in\Hc_N$, and $\Phi_1:=\{e_1, \cdots,e_1\}$, which is $M$ copies of $e_1$. Then $f(\Phi_1)=1$. Choose $\Psi$ that is close enough to $\Phi_1$ but spans $\Hc_N$ (a frame), then we have $f(\Psi)$ is close to $f(\Phi_1)=1$.
\end{proof}

\section{Minimizing with Frobenius norm}\label{sec:fro}
This section focuses on Problem \eqref{equ:fro}. Notice for the Frobenius norm, the optimal operator $T_{\Phi}$ is the projection of $I_N$ onto the cone $C_{\Phi}=\{\sum_{i=1}^Mc_i\varphi_i\varphi_i^*:c_i\geq0\}$. Since $C_{\Phi}$ is closed and convex, we have the uniqueness of $T_{\Phi}$. To remind the reader of the notation, let \begin{equation*}
\min_{c_i\geq0}\|I_N-\sum_{i=1}^Mc_i\varphi_i\varphi_i^*\|_F=\|I_N-T_{\Phi}\|_F.
\end{equation*}

%

\subsection{Invertibility of $T_{\Phi}$}\label{sec:invert}
We are still interested in the condition number of the optimal operator $T_{\Phi}$. Surprisingly, $T_{\Phi}$ does not need to be invertible, or equivalently, $T_{\Phi}$ may not be a frame operator.
We need some basic facts to set up counterexamples.

Let \begin{equation}\label{equ:PI}
\min_{c_i\in\R }\|I_N-\sum_{i=1}^Mc_i\varphi_i\varphi_i^*\|_F=\|I_N-P_{\Phi}\|_F.
\end{equation}
Notice we allow $c_i$ to be negative here, so $P_{\Phi}$ is the projection of $I_N$ onto the subspace spanned by $\{\varphi_i\varphi_i^*\}_{i=1}^M$. Therefore  $P_{\Phi}\neq T_{\Phi}$ in general.

Observe that
\begin{align}
h(c)=\left\|\sum_{i=1}^Mc_i\varphi_i\varphi_i^* - I_N\right\|_F^2&= \tr\left(\sum_{i,j=1}^Mc_ic_j\varphi_i\varphi_i^*\varphi_j\varphi_j^* - 2\sum_{i=1}^Mc_i\varphi_i\varphi_i^* + I\right)\\
&=\sum_{i,j=1}^Mc_ic_j|\langle\varphi_i,\varphi_j\rangle|^2 - 2\sum_{i=1}^Mc_i\|\varphi_i\|_2^2 + N.
\end{align}

Taking the partial derivative with respect to $c_i$, we get
$$\frac{\partial h}{\partial c_i}=2\sum_j c_j|\langle \varphi_i,\varphi_j\rangle|^2-2\|\varphi_i\|_2^2$$
Therefore if $P_{\Phi}=\sum_{i=1}^Mc_i\varphi_i\varphi_i^*$, then the scalars  should satisfy
\begin{equation}\label{equ:linear}
Fc=g,
\end{equation}
where $F$ is as defined in Section \ref{sec:note} as $F(i,j)=|\langle \varphi_i,\varphi_j\rangle|^2$, and 
$g=(\|\varphi_1\|_2^2,\cdots,\|\varphi_N\|_2^2)^*
$. Since $F$ is the Gram matrix of $\{\varphi_i\varphi_i^*\}_{i=1}^M$, $F$ is strictly positive definite if $\{\varphi_i\varphi_i^*\}_{i=1}^M$ is linearly independent, in which case $c=F^{-1}g$ is the global minimum of \eqref{equ:PI}. Furthermore, if it so happens that the coordinates of $c=F^{-1}g$ are all nonnegative, then naturally $P_{\Phi}=T_{\Phi}$.

Equation \eqref{equ:linear} is not a necessary condition for $c$ to be an optimal scaling because we are minimizing $h(c)$ over the first orthant (rather than over the whole space). But given an optimal scaling, we still have $\frac{\partial h}{\partial c_i}=0$ for $i\in\{j:c_j>0\}$ since it does not sit on the boundary of the first orthant. To summarize,

\begin{proposition}\label{pro:gradient} Given $\Phi\in\Fc(M,N)$,
\begin{itemize}
\item[(1)] If $\{\varphi_i\varphi_i^*\}_{i=1}^M$ is linearly independent, then $c=F^{-1}g$ is the unique solution of \eqref{equ:PI}. If further $c=F^{-1}g\geq0$ (component wise), $c$ is the unique solution of both \eqref{equ:fro} and \eqref{equ:PI}, and $P_{\Phi}=T_{\Phi}$.
\item[(2)] If $c=(c_i)_{i=1}^M$ is an optimal scaling, then
\begin{equation}\label{equ:partial}\sum_{j=1}^M c_j|\langle \varphi_i,\varphi_j\rangle|^2=\|\varphi_i\|_2^2, \quad \forall i\in\{j:c_j>0\}
\end{equation}
\end{itemize}
\end{proposition}

\begin{example}
For $\Phi=\begin{bmatrix}
1&1&1\\
0&1&0\\
0&0&1
\end{bmatrix}$, the optimal operator $T_{\Phi}$ is not invertible. Because if $T_{\Phi}$ is invertible, the optimal scaling must have $c_i>0, i=1,2,3$. By Proposition \ref{pro:gradient}(2), it must hold that the solution of \eqref{equ:linear} is all positive. However, solving \eqref{equ:linear} gives
$ c=(-1,1,1).$ 

\end{example}

\begin{example} \label{exa:bad}
 In $\R^N (N\geq3)$ we define

\[ \varphi_1 = (1,0,\ldots,0),\ \varphi_2 = (a,b,0,\ldots,0)
\mbox{ where } a^2+b^2=1 \mbox{ and } a<b,\]
and
\[ \varphi_i = (c,c,0,\ldots,\underbrace{\sqrt{1-2c^2}}_{i\text{th}},0,\ldots,0),
\mbox{ where } c^2+c^2(a+b)^2 = 1+a^2, i\geq3.\]
It is easy to check that $2c^2 <1$.  
Now, $\{\varphi_i\}_{i=1}^N$ is a frame of $\R^N$. We display the first two columns of
$F$:
\[ F=\begin{bmatrix}
1& a^2& \cdots\\
a^2 & 1 & \cdots\\
c^2 & c^2(a+b)^2 & \cdots\\
\vdots & \vdots & \vdots\\
c^2&c^2(a+b)^2 & \cdots
\end{bmatrix} 
\]

Once again, if $T_{\Phi}$ were invertible, we require $c_i>0$, and $Fc=g$. 
However, the solution of $Fc=g$ is 
\[ c=(\frac{1}{1+a^2},\frac{1}{1+a^2},0,\ldots,0),\]
a contradiction.
\end{example}

\begin{remark}
The example above can be generalized to arbitrarily many frame vectors. Indeed, define $\varphi_i, i=1,2,\cdots,N$ as in Example \ref{exa:bad}, and add in $\varphi_{N+1},\cdots, \varphi_M$ such that $\varphi_i\varphi_i^*\in\{\sum_{i=1}^Nc_i\varphi_i\varphi_i^*: c_i\geq0\}, i=N+1, \cdots, M$.
\end{remark}

The following proposition shows that in $\R^2$, the optimal operator is always invertible.

\begin{proposition}\label{prop:R2}
Given any frame $\Phi\in\Fc(M,2)$, the optimal operator $T_{\Phi}$ for the Frobenius norm is invertible.
\end{proposition}

\begin{proof}
Suppose the eigenvalues of $T_{\Phi}$ are $\lambda_1\geq\lambda_2$, then
$\|I_N-T_{\Phi}\|_F^2=(1-\lambda_1)^2+(1-\lambda_2)^2$. In order to prove $\lambda_2>0$, it suffices to show that $\|I_N-T_{\Phi}\|_F<1$.

Pick two independent vectors, say $\varphi_1, \varphi_2$ from the frame. It suffices to show that $$\min_{c_i\geq0}\|I_N-c_1\varphi_i\varphi_i^*-c_2\varphi_2\varphi_2^*\|_F<1,$$ since $\|I_N-T_{\Phi}\|_F\leq \min_{c_i\geq0}\|I_N-c_1\varphi_i\varphi_i^*-c_2\varphi_2\varphi_2^*\|_F$.

With rotation and adding a negative sign to the vectors, we can assume without loss of generality that
\[ \varphi_1 = (a,b)\mbox{ and } \varphi_2 = (a,-b).\]
We can further assume that they are both unit norm as $a^2+b^2=1$ and $a\geq b$.

A direct calculation shows that the eigenvalues of the
frame operator $S$ of $\{\varphi_1,\varphi_2\}$ are  $\{2a^2,2b^2\}$. 

$$\min\|I-\sum_{i=1}^2 c_i \varphi_i\varphi_i^*\|_F^2\leq \min_{c\geq0}\|I-cS\|_F^2=\min_{c\geq0}\left((1-2ca^2)^2+(1-2cb^2)^2\right)$$

Taking $c=\frac{1}{2a^2}$, we get
$$\min\|I-\sum_{i=1}^2 c_i \varphi_i\varphi_i^*\|_F^2\leq (1-b^2/a^2)^2<1.$$
\end{proof}

\subsection{When is $T_{\Phi}$ invertible?}

It is not easy to find a condition on the frame so that its optimal operator is guaranteed to be at least invertible. But the situation can be simplified when the frame $\Phi$ is full spark (every $N$ frame vectors from $\Phi$ span $\Hc_N$), and the outer products $\{\varphi_i\varphi_i^*\}_{i=1}^M$ are linearly independent. In this case, if the solution $c=F^{-1}g$ happens to be 
$$c_i\geq0, \text{ and } \#(\mathrm{supp}(c))\geq N,$$
then we are guaranteed to have an invertible $T_{\Phi}$ because the scaled frame $\{\sqrt{c_i}\varphi_i\}$ will for sure span. Moreover, the condition that both $\Phi$ is full spark and $\{\varphi_i\varphi_i^*\}$ is independent is not harsh when $M$ is no greater than the real dimension of the $N\times N$ symmetric matrices ($N^2$ for $\C$, and $N(N+1)/2$ for $\R$). In fact, such sets of frames are generic.

The following theorem also provides a sufficient condition for the invertibility of $T_{\Phi}$, when the outer products are independent. Moreover, the advantage of this sufficient condition is that it is directly on the frame.

\begin{theorem}\label{thm17}
Let $\{\varphi_i\}_{i=1}^M$ be a unit norm frame for $\Hc_N$  and let $\{\varphi_i\varphi_i^*\}_{i=1}^M$ be linearly independent. Let $f_i$ be the $i$th column of $F=\left(|\langle\varphi_i,\varphi_j\rangle|^2\right)$. 
If 
\[| \|f_i\|_1 - \|f_j\|_1|< \lambda\quad \forall i,j=1,2,\ldots,M,
\]
where $\lambda$ is the smallest singular value of the the matrix $F$,
then
$$T_{\Phi}=P_{\Phi} = \sum_{i=1}^M a_i\phi_i\phi_i^*\mbox{ where } a_i > 0, \mbox{ for all } i=1,2,\dots,M.$$

%
\end{theorem}

\begin{proof}
By Proposition \ref{pro:gradient}, it suffices to show that the solution of $Fa=g$ has positive coordinates. This is equivalent to showing that $\det F_i>0$, where $F_i$ denotes the matrix $F$ with its $i$th column replaced by $g$, since $\det F>0$.

First note that if
\[ \delta= \frac{1}{2}\left ( \max_{1\le i \le M}\|f_i\|_1 + \min_{1\le i \le M}\|f_i\|_1\right ),\]
then
\[ \left| \|f_i\|_1 - \delta\right|< \lambda.\]

Let $D= \left ( \sum_{j=1}^Mf_j\right ) -\delta g$, then
\begin{eqnarray}
\det F &=& \det (f_1,f_2,\ldots,f_{i-1},\sum_{j=1}^Mf_j,f_{i+1},\ldots,f_M)\\
&=& \det(f_1,\ldots,f_{i-1},\delta g,f_{i+1},\ldots,f_{M}) +\det (f_1,\ldots,f_{i-1},D,f_{i+1},\ldots,f_M)\\
&=& \delta\cdot  \det F_i + \det(f_1,\ldots,f_{i-1},D,f_{i+1},\ldots,f_M).\label{equ:Fi}
\end{eqnarray}

Note that if $D=(d_1,d_2,\ldots,d_M)^T$,
then 
\[ |d_i|=\left|\sum_j|\langle \varphi_i,\varphi_j\rangle|^2- \delta\right|< \lambda,\mbox{ for all } i=1,2,\ldots, M.\]

Let $M_{ij}$ be the determinant of the matrix obtained by deleting the $i$th row and $j$th column of $F$. For fixed $i$, let
\[ D^{(i)}=(d_1^{(i)},d_2^{(i)},\ldots,d_M^{(i)})^T \mbox{ where } d_j^{(i)}(-1)^{i+j}M_{ij} = |d_j||M_{ij}|\mbox{ for all }j=1,2,\ldots,N.\]
Let $x^{(i)}=(x^{(i)}_1,x_2^{(i)},\ldots,x_M^{(i)})^T=F^{-1}D^{(i)}$ and  by Cramer's rule
\begin{eqnarray}
 (\det F)x^{(i)}_i &=& \det(f_1,f_2,\ldots,f_{i-1},D^{(i)},f_{i+1}, \ldots,f_M)\\
 &=& \sum_{j=1}^M(-1)^{j+i}d_j^{(i)} M_{ji}=\sum_{j=1}^N|d_j| |M_{ji}|\\
 &\ge& \left | \sum_{j=1}^M (-1)^{j+i}d_j M_{ji} \right |\\
 &=& |\det(f_1,\ldots,f_{i-1},D,f_{i+1},\ldots,f_N)|\label{equ:abs}
  \end{eqnarray}
On the other hand, since $\lambda$ is the smallest singular
value of $F$, 
 \begin{eqnarray}\label{equ:xi}
  x_i^{(i)} = (F^{-1}D^{(i)})_i\le \|F^{-1}D^{(i)}\|_{\infty}
  \le\|F^{-1}\|\|D^{(i)}\|_{\infty}
  < \frac{1}{\lambda}\lambda
\end{eqnarray}

By \eqref{equ:Fi}, \eqref{equ:abs} and \eqref{equ:xi},
\begin{align*}
\delta\cdot\det F_i&=\det F-\det(f_1,\cdots,f_{i-1},D,\cdots,f_M)\\
&\geq \det F-(\det F)x_i^{(i)}>\det F-\det F\cdot 1=0
\end{align*}

\end{proof}

\subsection{Convexity of $\mathcal{O}_{\Phi}$ revisited}\label{sec:conv}
The convexity of optimal scalings with the Frobenius norm is thoroughly studied in \cite{cc12}, in the case when frames are scalable. We wish to generalize it to any frames. Minimal optimal scalings play an important role in the structure of $\mathcal{O}_{\Phi}$, so we make some important observations of minimal optimal scalings first, which are missing in \cite{cc12}.

\begin{proposition}\label{pro:min} Given $\Phi\in\Fc(M,N)$,

(a) If $\{c_i\}_{i=1}^M$ is a minimal optimal scaling, then the outer products associated with it, i.e., $\{\varphi_i\varphi_i^*: c_i>0\}$, are linearly independent.

(b) Different minimal optimal scalings have different supports.
\end{proposition}
\begin{proof}
(a) Let $I=\{i: c_i>0\}$. $\{\varphi_i\varphi_i^*\}_{i\in I}$ is linearly independent if and only if there is only one point in the set $SS=\{(x_i)_{i\in I}:\sum_{i\in I}x_i\varphi_i\varphi_i^*=T_{\Phi}, x_i\in\R\}$.  Suppose to the contrary that $SS$ is a nontrivial affine subspace (more than one point). If we call $\{(x_i)_{i\in I}:x_i\geq0\}$ the positive orthant, then the affine subspace $SS$ intersects the positive orthant since we have a scaling to begin with. Therefore, it will also intersect the boundary of the positive orthant, providing a solution of $\{(x_i)_{i\in I}:\sum_{i\in I}x_i\varphi_i\varphi_i^*=T_{\Phi}, x_i\geq0\} $ with at least one $x_i$ to be 0, which contradicts to the fact that $\{c_i\}_{i=1}^M$ is minimal.

(b) If they have the same support $I$, then having two minimal scalings means that $\{\varphi_i\varphi_i^*\}_{i\in I}$ is linearly dependent, which is not true by (a).
\end{proof}

We already know that $\mathcal{O}_{\Phi}$ is convex. But for the Frobenius norm, it is a polytope, with minimal optimal scalings as its vertices. The  proof is similar to that of \cite{cc12}.
\begin{theorem}
With Frobenius norm, the set of optimal scalings $\mathcal{O}_{\Phi}$ is a polytope. Moreover, $\mathcal{O}_{\Phi}$ is the convex hull of the minimal optimal scalings, i.e., the vertices of $\mathcal{O}_{\Phi}$ are minimal optimal scalings.
\end{theorem}
\begin{proof}
$\mathcal{O}_{\Phi}$ is a polytope because it is an intersection of half planes: $\mathcal{O}_{\Phi}=\{(x_i)_{i=1}^M:\sum x_i\varphi_i\varphi_i^*=T_{\Phi}, x_i\geq0\}$.

 To prove the vertices part, we first  show any vertex must be a minimal optimal scaling of $\mathcal{O}_{\Phi}$.  Let $u\in \mathcal{O}_{\Phi}$ be a vertex, and assume to the contrary that $u$ is not minimal.  Then there exists $v\in \mathcal{O}_{\Phi}$ whose support is a proper subset of $u$'s. Let $w(t)=v+t(u-v)$. We observe that $w(t)\in\mathcal{O}_{\Phi}$ if and only if every component of $w(t)$, $w(t)_i\geq0$. Pick
 $$t_0=\left\{\begin{array}{ll}
 2, & \text{If for every $i$, $v_i\leq u_i$}\\
 \mathrm{min}\{\frac{v_i}{v_i-u_i}:v_i>u_i\}, & \text{otherwise}
 \end{array}\right..$$
We observe that  $t_0>1$ and $w(t_0)_i\geq0$, so $w(t_0)\in \mathcal{O}_{\Phi}$, which indicates that $u$ lies on the line segment with endpoints $v$ and $w(t_0)$, hence not a vertex. This is a contradiction.

Now suppose we are given a minimal scaling $w$ which is not a vertex of $\mathcal{O}_{\Phi}$.  Then we can write $w$ as a convex combination of vertices, say $w=\sum t_iv_i$, where we know at least two $t_i$'s are nonzero, say $t_1$ and $t_2$.  Since both $t_1$ and $t_2$ are positive and all the entries of $v_1$ and $v_2$ are nonnegative, it follows that $\mathrm{supp}(v_1)\cup\mathrm{supp}(v_2)\subseteq\mathrm{supp}(w)$. Moreover, $\mathrm{supp}(v_1)\neq\mathrm{supp}(v_2)$ by Proposition \ref{pro:min}(b). This contradicts to the fact the $w$ is a minimal scaling.
\end{proof}

The theorem above does not work for the operator norm case because the optimal operator is not necessarily unique.

\section*{Acknowledgements}

We would like to thank the reviewers for their constructive comments and thorough review, especially on the remark to allow condition number to be infinity.

\end{document}